\documentclass[a4paper,10pt, draft]{article}
\usepackage{graphicx, rotating, amsmath, amssymb, amsthm, nicefrac}

\newtheorem{teo}{Theorem}[section]

\newtheorem{cor}{Corollary}[section]
\newtheorem*{ex}{Example}

\newcommand{\rea}{\mathbb{R}}

\newcommand{\com}{\mathbb{C}}

\begin{document}
\title{Systematic Discovery of Runge-Kutta Methods through Algebraic Varieties}



\author{
Ivan Martino 
\footnote{Department of Mathematics -- Northeastern University, USA,  \emph{i.martino@northeastern.edu} }  
\and Giuseppe  Nicosia
\footnote{Department of Mathematics \& Computer Science -- University of Catania, Italy, \emph{nicosia@dmi.unict.it}}
 }

\maketitle

\begin{abstract}
This work  presents a new evolutionary optimization algorithm in theoretical mathematics with important applications in scientific computing.  
The use of the evolutionary algorithm is justified by the difficulty of the study of the parameterization of an algebraic variety, an important problem in algebraic geometry. 
We illustrate an application, {\sc Evo-Runge-Kutta}, in a problem of numerical analysis. 
Results show the design and the optimization of particular algebraic variety, the explicit $s$ levels Runge-Kutta methods of order $q$.
The mapping between algebraic geometry and evolutionary optimization is direct, and we expect that many open problems will be modelled in the same way.
\end{abstract}


\section{Introduction}
In science and engineering, problems involving the solution of a polynomial system appear everywhere. 
In this work we study the optimization of a positive real values function $\mathcal{G}$ defined over $X\subset \rea^N$, the set of real solutions of a system of $m$ polynomial equations $f_i(\mathbf{x})=0$ for $i=1,\dots, m$.
%
 
The key problem of optimizing $\mathcal{G}$ over $X$ is the necessity to control $X$ via its parameterizations: for reasonably large values of $N$ and $m$, it is an open problem to find the connected components $\{U_i\}$ of $X$, their local dimensions $d_i$ and their local parameterizations $\{\theta_i: \rea^{d_i}\rightarrow U_i\}$. 
Evidences do not suggest the use of symbolic computation or numerical method \cite{application_genetic_to_polynomial}.
For this reason, we state a particular methodology based on Evolutionary Algorithms.

The Evolutionary Algorithms (EAs) are a generic population-based meta-heuristic optimization algorithm bio inspired \cite{goldberg:02}. 
They \emph{evolve} a population of candidate solutions for the problem by \emph{reproducing} and \emph{selecting} with respect to a fitness function. 
They recently improve on several field because it is not ne\-ces\-sary any assumption about the solutions space of the problem. In particular the evolutionary algorithms fit very well into the optimization problem where the fitness function is cheap.
In this research field, the EAs were used to solve a particular class of polynomial system in \cite{application_genetics_to_system} and an hybrid algorithm was presented with the same aim in \cite{evo_alg_2}. In addition there were genetic approaches in pure algebra in \cite{application_genetic_to_polynomial} and \cite{finite-algebra}.

In applications it is not usual have a proper justification for the use of such optimization algorithms, apart cases where the dimension of the problem is prohibitive for any numerical computation.
We show that algebraic geometry provides a justification for why it is important to use evolutionary optimization algorithm to optimize any positive real function over the algebraic variety $X$, but this justification holds for any kind of optimization over $X$, \cite{two-stage} \cite{coello:07}, \cite{deb:01}.

In Section \ref{sec-alg-var} we introduce some notions about the algebraic varieties for a non-expert reader. 
In Section \ref{sec-optimization-problem}, we set up all the theoretical objects used in the most general environment and after that we present the abstract algorithm {\sc Evo-Alg-Variety}. 

We apply our suggested methodology in a numerical analysis interesting problem: in \cite{esempio1}, \cite{esempio2}, one wants to perform some particular features of the RK methods. With {\sc Evo-Runge-Kutta}, a new evolutionary optimization algorithm designing Runge Kutta methods, we find new Runge-Kutta methods of order $q$ with minimal approximation error.

At first, we introduce the Runge-Kutta methods in Section \ref{sec:RK-def}, sketching some important results in \cite{Hai}. In Section \ref{algebraic-variety}, we apply them to define the problem from a geometrical point of view. We show that
\[
	V_{q}^s=\left\lbrace \mbox{$s$-level Runge-Kutta methods of order accuracy $q$}\right\rbrace
\]
is an algebraic variety and one cannot compute its dimension and its local parameterizations. Section \ref{results} shows the obtained results.

The final section discusses the conclusions of the work underlining the possible directions of future optimization algorithms and the implications for algebraic varieties and pure algebra.

\section{The Algebraic Varieties}\label{sec-alg-var}

In this section we sketch some notions about the algebraic varieties for the reader without this background; we suggesting to an expert in algebraic geometry to skip this section. 

Claiming a zero set of a polynomial system being an algebraic varieties has some underlying effects. 
One of the main differences concerns the topology. 
When we write $\rea^k$, we normally mean the set of $k$-tuples of real numbers with the standard (or Euclidean) topology, i.e. the topology in which we consider the open sets as the the $k$-dimensional open ball, $B(x_0,\,\epsilon)=\{x\in \rea^k:\, |x-x_0|<\epsilon\}$, of center $x_0\in \rea^k$ and radius $\epsilon>0$.
Instead, $\mathbb{A}^{k}(\rea)$ is the affine $k$-dimensional space: the set of $\rea$-points is the same as before, but the topology is radically different; here, as in any algebraic variety, it is used the Zariski topology: a closed subset (the complement of open subset) is the zero set of some polynomials in $k$ variables. In contrast to the standard topology, the Zariski topology is not Hausdorff\footnote{%
A topological space $Y$ is Hausdorff if it is possible to separate two distinct points with two disjoint open sets.}. %
This non-trivial example shows the difference between the two topologies.

Since the objects (the varieties) are different, so the maps among those will be. Let us start, as previous, with the standard case and let $f(x)=\frac{1}{x}$ be a real function on an open subset of $\rea$; the function is not defined for $x=0$ and, similarly, the zero value is not in the image of $f$. This happens because the set $\{0\}$ is a closed point. 
In the Zariski topology there are a lot of closed sets more; in particular algebraic map can be undefined over a finite number of those. 
For the reader who meets algebraic maps for the first time, it is useful to know that such maps have not uniquely way to be written. 
For instance, let $X$ be the algebraic variety given by $y^2=x^3+x^2$ in $\mathbb{A}^{2}(\rea)$ and let $f:X\rightarrow \rea$ be $f(x,y)=(\nicefrac{y}{x})^2$; it seems that $f$ is not defined in $x=0$, but using the defining polynomial of $X$, one has
\[
	f(x,y)=\left( \frac{y}{x}\right)^2=\frac{y^2}{x^2}=\frac{x^3+x^2}{x^2}=x+1.
\]
So, for $x=0$, $f$ assumes value $1$ and $f$ can be written either as $(\nicefrac{y}{x})^2$ or as $x+1$. I refer to \cite{Hart} for the undefined geometrical notions in this section and in the next sections.

\subsection{Local Dimension and Local Parameterization}\label{sec-dimension}
The first natural question regarding the geometry of an algebraic variety $V$ is about its dimension, or, with different words, the number of free parameters that we need to use to parameterize it, globally or locally. 

A very nice introduction to dimension theory is in Chapter II.8 of \cite{Eis}, in which there are equivalent definitions for the dimension of an algebraic variety: for example using the \emph{Noether Normalization Lemma} (see \cite{AM}). However, there exists a \emph{finite}\footnote{%
An open subset $U$ of a scheme $X$ is affine if there exists a ring $R$ such that $U=Spec(R)$; the Zariski topology, that we have discussed before, comes from the notion of $Spec$ of a ring. A map between algebraic schemes, $f:X\rightarrow Y$, is a finite map if there exists an open affine covering of $Y$,  $\{U_i=Spec(B_i)\}_{i\in I}$, such that for each $i$, $f^{-1}(U_i)=Spec(A_i)$ where $A_i$ is a finite $B_i$-algebra.
}%
map $\varphi: V \rightarrow \mathbb{A}^d(\rea)$, where $d$ is the dimension of $V$. 
It is extremely difficult to compute $d$ (see \cite{hardness}, \cite{num-sim}): there are some methods in computational algebra where the complexity of the computation depends on the number of generators $m$, the number of variables $N$ and the degree of the polynomials $q$ as $m^{o(1)}q^{o(N)}$ (see \cite{comp-dim}, \cite{comp-dim2}). 
Hence, it is clear that with $N$ and $m$ bigger enough any symbolic approach to this problem fails.

Moreover, since we work with real numbers, there is an additional complication: working with the complex field guarantees that the varieties are connected. 
For instance the circle $\mathcal{C}_{\com}$ given by $x^2+y^2=-1$ is a one-dimensional variety in the affine complex plane and, if we compactify it, one knows that $\mathcal{C}_{\com}$ is projectively equivalent to (the compactification of) the hyperbola $\mathcal{I}_{\com}$, defined by $xy=1$. 
Instead, working in the real field, $\mathcal{C}_{\rea}$ is the empty set and $\mathcal{I}_{\rea}$ is non-connected: $\mathcal{I}_{\rea}$ is a union of two pieces $U_1=\{xy=1, x,y\geq 0\}$ and $U_2=\{xy=1, x,y\leq 0\}$. Trivially, $\mathcal{C}_{\rea}$ and $\mathcal{I}_{\rea}$ are far to be equivalent, just as sets. 

Thus, an affine real variety $V$ should be seen as a union of connected components, $V=\cup_{i\in I} U_i$; each one of this component $U_i$ has a proper dimension $d_i$ with $0\leq d_i\leq d$. 

In order to perform any kind of optimization over an algebraic variety $V$, we need to parameterize it. In fact, we need to control the set $V$ using free parameters $\{p_1,\dots, p_k\}$ where $k$ is at least the dimension $d$ of $V$. 
More explicitly, given a connected component decomposition, $V=\cup_{i\in I} U_i$, we need to know the functions
\[
	\theta_{i}: \rea^{d_i} \rightarrow U_i,
\]
for all $i\in I$. This problem, translated in the algebraic geometry form, 
\[
	\theta_{i}: \mathbb{A}^{d_i}(\rea) \rightarrow U_i,
\]
is the problem of finding the \emph{local parameterization of an real algebraic variety}. I remark, that, even if the meaning and the appearance of the two $\theta_{i}$ are exactly the same, the algebraic one carries all the differences discussed.

If we narrow down our investigation, for a moment, and consider only curves (one dimension algebraic varieties) in the affine plane (see \cite{fulton}, \cite{curves:96}), the que\-stion is the following: \emph{is the curve rational?} We know that non-rational curves exist. So not all varieties $X$ are birational equivalent (generalization of rational concept) to $\mathbb{P}^{d}(\rea)$. The theory of Gr\"{o}ebner basis (see \cite{Cox}) or the Newton Polygon (see for recently connection with tropical geometry \cite{AS}) could be applied to tackle this problem, but, with a lot of generators, the computational time is not acceptable.

The problem is complicated, hence, excluding some particular case, finding local parameterizations of an algebraic variety generated by a large number of polynomials in a large affine space is an open problem (see \cite{openproblems}).

\section{Optimization over an Algebraic Variety}\label{sec-optimization-problem}
In this section, we state the optimization problem and the suggested algorithm in the most abstract description. 

Let $X\subset \rea^N$ be the set of real solutions of a system of $m$ polynomial equations
\begin{equation*}
	\left\lbrace  	\begin{array}{l}
 				f_1=0,\\
				f_2=0,\\
				\vdots\\
				f_m=0.
			\end{array}
	\right.
\end{equation*}
Let $N$ and $m$ be great enough. Let $\mathcal{G}$ be a positive real values function defined over $X$ and let us consider the optimization problem consisting of finding $\mathbf{x}\in X$ minimizing the value $\mathcal{G}(\mathbf{x})$ in $\mathcal{G}(X)\subset \rea_+$. 

In applications (see for examples \cite{regular_system}, \cite{complete_algorithm}), it is usual subdividing the generators in a finite number of sets. Let $\{f_i\}_{i=k_{j-1},\dots, k_j}$ be this subdivision in $c$ sets (with $k_0=1$ and  $k_c=m$). 
We define the algebraic varieties 
\[
	X_j=\{x\in \mathbb{A}_{\rea}^N:\, f_i(x)=0,\,\, \forall i=1,\,\dots,\, k_j\},
\]
for $j=1,\dots, c$ and $X_0=\mathbb{A}_{\rea}^N$ and $X_{c}=X$. Hence, there is the following chain of $c$ inclusions:
\[
	\mathbb{A}_{\rea}^N=X_0 \supseteq X_1 \supseteq \dots \supseteq X_{c-1} \supseteq X_c=X;
\]
if $\{f_i\}_{i=1,\dots, m}$ is a minimal basis for the ideal (see \cite{Cox}) of $X$, then inclusions are strict.
 
If there is no motivation to make a subdivision of the generators, then it is possible to choose, for $0<i<v$, $k_i=i\lfloor\nicefrac{m}{v}\rfloor$, with $v$ a suitable positive integer: the length of the chain will be exactly $c=v+1$. One should choose $v$ coherently with the hardness of solving the system of polynomials.

Applications (see \cite{weight_generator}) also suggest that some constrains $f_i(x)=0$ are more important than the others. So, let $\mathbf{r}=(r_i)_{i=1,\dots, m}$ be an element of the simplex
\[
    \Delta_m=\{(r_i)_{i=1,\dots, m}: r_i>0, \sum_{i=1}^m r_i=1\};
\]
$r_i$ is the weight of the polynomial equation $f_i(x)=0$. Thus, for $j=1,\,\dots,\, c$ the positive real value function $\mathcal{H}_j:\mathbb{A}_{\rea}^N\rightarrow \rea_+$ is defines as
\[
	\mathcal{H}_j(x)=\sum_{i=1}^{k_j} r_i|f_i(x)|.
\]
Let $\epsilon=\{\epsilon_1 < \epsilon_2 < \dots < \epsilon_c\}$ be a set of positive non increasing real numbers; we define the $\epsilon$-tubular neighbourhood of $X_j$ being
\[
	\epsilon(X_j)=\{x\in \mathbb{A}_{\rea}^N: \mathcal{H}_k(x)<\epsilon_k, k=1,\dots, j\}.
\]

Finally, we propose the algorithm in Table \ref{alg:general} and we describe the objects used in Table \ref{tab:para-general}.
\begin{table}[htbp]
\begin{tabbing}
{\bf \sc Evo-Alg-Variety}($N$, $\{f_i\}$, $\{k_j\}$, $\mathcal{G}$, $\mathbf{r}$, $\epsilon$, $\{S_j\}$)\\
1. $S_0=\emptyset$;\\
2. {\bf For} 	\= ($y$ from $0$ to $c-1$)\\
3. 	\> EA-minimize ($\mathcal{H}_{y+1}$ over $\epsilon(X_y)$ starting from $S_y$);\\
4. 	\> save-good-solution ($S_{y+1}$);\\
5. {\bf End for};\\
6. EA-minimize ($\mathcal{G}$ over $\epsilon(X)$ starting from $S_c$);
\end{tabbing}
\caption{The Algorithm for the optimization of $\mathcal{G}$ over the algebraic variety $X$}
\label{alg:general}
\end{table}

The kernel of the optimization is the fact that we sequentially produce an homogeneous covering of the affine varieties $X_j$: in each $y$-cycle we optimize $\mathcal{H}_{j+1}$ on $\epsilon(X_j)$. In other words, the fitness function of every $y$-cycle coincides with the constrain of $\epsilon(X_{j+1})$.

For this reason, in addition, we save the points $\tilde{x}$ such that $\mathcal{H}_{j+1}(\tilde{x})<\epsilon_{j+1}$, i.e we save the elements $\tilde{x}\in \epsilon(X_j)$ such that $\tilde{x}\in \epsilon(X_{j+1})$. Those feasible solutions are the suggested starting points for the evolutionary algorithm of the next cycle.

These information can be used for any optimization over the variety $X$. In fact, in the end of the algorithm, we run the last optimization over the feasible points of $\epsilon(X_c)=\epsilon(X)$ minimizing the function $\mathcal{G}$.

\begin{table}[htbp]
\begin{center}
\begin{tabular}{l|l}
	Object & Description\\ \hline \hline
 	$N$ & number of variables\\
	$m$ & number of equations\\
	$\{f_i\}$ & polynomial equations\\
	$\{k_j\}$ & index of the subdivision\\
	$c$ & number of the sets subdivision\\
	$\mathcal{G}$ & function to optimize\\
	$\mathbf{r}$ & weights for the polynomial equations\\
	$\epsilon$ & numbers defining the $\epsilon$-tubular neighbourhood $\epsilon(X_j)$ of $X_j$\\	
	$S_j$  & set of feasible solutions for $\epsilon(X_j)$
\end{tabular}
\caption{Objects used in the algorithm described in Table \ref{alg:general}}
\label{tab:para-general}
\end{center}
\end{table}

Since we work in a real field and since the $N$ and $m$ are large enough to make impossible the analytic computation of the local parameterizations $\theta_{i}: \rea^{d_i} \rightarrow U_i$ (with $X=\cup_{i\in I} U_i$), then the use of the EAs is fully justified.

What fallows it is devoted to an application of this abstract algorithm to a numerical analysis open problem.

\section{The Runge-Kutta Methods}\label{sec:RK-def}
Problems involving ordinary differential equations (ODEs) can be always reformulated as set of $N$ coupled first-order differential equations for the functions $y_i$, having the general form 
\begin{equation}\label{eq1}
\frac{dy_i(t)}{dt} = g_i(t, y_1, \dots, y_{N})\, \text{ with } i=1, 2, \dots, N.
\end{equation}
A problem involving ODEs is  completely specified by its equations and by boundary conditions of the given problem. Boundary conditions are algebraic conditions: they divide into two classes, \emph{initial value problems} and \emph{two-point boundary value problems}.
In this work will consider the initial value problems, where all the $y_i$ are given at some starting value 
$t_0$, and it is desired to find the $y_i$'s at some final point $t_f$.
In general, it is the nature of the boundary conditions that determines which numerical methods to use.
For instance the basic idea of the \emph{Euler's method} is to rewrite the $dy$'s and $dt$'s in (\ref{eq1}) as finite steps $\delta y$ and $\delta t$, and multiply the equations by $\delta t$. This produces algebraic formulas $\delta y$ with the independent variable $t$ increased by one \emph{stepsize} $\delta t$; for very small stepsize a good approximation of the differential equation is achieved. 

The Runge-Kutta method is a practical numerical method for solving initial value problems for ODEs \cite{Hai}.
Runge-Kutta methods propagate a numerical solution over an $N+1$-dimensional interval by combining the information from several Euler-style steps (each involving one evaluation of the right-hand $g$'s), and then using the information obtained to match a Taylor series expansion up to some higher order. Runge-Kutta is usually the fastest method when evaluating $g_i$ is cheap and the accuracy requirement is not ultra-stringent.

In what follows, we summarize some basic notions about Runge Kutta me\-thods (RKms) and we remark the results that we are going to use.

Without lost of generality we can consider, instead of (\ref{eq1}), the autonomous system
\begin{equation}\label{eq:autonomous}
	\left\{
	\begin{array}{l}      
	y'=g(y) \\ 
	y(t_0)=y_0
	\end{array}
	\right.
\end{equation}
given by the function $g:\Omega \rightarrow \rea^n$ (with $\Omega$ being an open subset of $\rea^{n}$) such that the Cauchy problem makes sense. Moreover, we suppose that $g$ satisfies the hypothesis of Schwartz' Theorem for mixed derivatives of each order that we will use.

The Runge-Kutta methods are a class of methods to approximate the exact solution of (\ref{eq:autonomous}). The structure of $s$-level \emph{implicit} RKm is the $i$-stage value
\begin{eqnarray*}
	k_i&=&g(y_0+h\sum_{j=1}^{s}a_{i,j}k_j),\\
\end{eqnarray*}
for $i=1,\,\dots,\, s$. The numerical solution $y_1\in \rea^n$ of the problem (\ref{eq:autonomous}) is 
\[
	y_1=y_0+h\sum_{i=1}^{s}w_i k_i;
\]
where $h$ is the step size. In the implicit case it is necessary to invert (almost always numerically) the function $g$ to reach the $k_i$ values. To avoid that, the \emph{explicit} RKms are often used: $a_{i,j}=0,\, \forall i\geq j$ and so the $k_i$ are obtained sequentially.
Despite this nice feature, there are some problems (the stiffness problems) where only \emph{implicit} RKms are required.

The Butcher Tableau \cite{Butcher:08}, in Tables \ref{tab:BT}, shows all the used parameters: the $w_i$ and the $a_{i,j}$ are real numbers and characterize a given method with respect another one. A RKm for a non-autonomous system involves also the $c_i$ para\-meters regulating the step size in the time axis. In the autonomous restriction such a parameters can be avoided because they fulfil the following \emph{autonomy conditions}:
\begin{equation}\label{RK-condition-autonomo}
	c_i=\sum_{j=0}^{i-1}a_{i,j},\, \forall i=1,\,\dots,\, s.
\end{equation}

\begin{table}[htbp]
\begin{center}
\begin{tabular}{c|cccccc}
	$c_1$ & $a_{1,1}$		& $a_{1,2}$	&$a_{1,3}$		& $\cdots$ &$a_{1,s-1}$& $a_{1,s}$ \\ 
	$c_2$ & $a_{2,1}$ 		& $a_{2,2}$ 	&$a_{2,3}$		& $\cdots$ &$a_{2,s-1}$& $a_{2,s}$ \\
	$c_3$ & $a_{3,1}$		& $a_{3,2}$ 	&$a_{3,3}$		& $\cdots$ &$a_{3,s-1}$& $a_{3,s}$ \\ 
	$\vdots$ & $\vdots$ 		& $\vdots$ 	&$\vdots$		&  $\ddots$&$\vdots$& $\vdots$ \\
	$c_{s-1}$ & $a_{s-1, 1}$ 		& $a_{s-1,2}$	&$a_{s-1,3}$		& $\cdots$& $a_{s-1,s-1}$& $a_{s-1,s}$\\
	$c_s$ & $a_{s,1}$ & $a_{s,2}$	&$a_{s,3}$	& $\cdots$	& $a_{s,s-1}$&$a_{s,s}$\\ \hline
	 &$w_1$ 			&$w_2$			&$w_3$			& $\cdots$ &$w_{s-1}$		& $w_{s}$\\
\end{tabular}
\end{center}
\caption{A Butcher Tableau for an $s$-level implicit Runge-Kutta method.}
\label{tab:BT}
\end{table}

The order of approximation of a RKm is the order of accuracy with respect to $h$ of the \emph{local truncation error},
\[
	\sigma_1(h)=y(t_0+h)-y_1(h),
\]
measuring how much the numerical solution $y_1(h)$ matches with the exact solution $y(t_0+h)$.

To understand the relation between its parameters and its order of approximation, in \cite{Hai} it is developed a combinatorial interpretation of the equations that the parameters of a RKm must satisfy (so called \emph{order condition equations}): there exits a bijection between the set of $p$-elementary differentials and the set, $T_p$, of rooted labelled trees of order $p$. 
Moreover, the following theorem shows how to obtain the condition equations of order $q$ of an $s$ levels RKm.

\begin{teo}\label{theo-order-condition}
	A Runge-Kutta method defined by a Butcher Tableau is of order $q$ if and only if
	\begin{equation}\label{constrain}
		\sum_{j=1}^s w_j \Phi_j(t)=\frac{1}{\gamma(t)}
	\end{equation}
	for all the rooted labelled trees $t$ of order at most $q$.
\end{teo}

Using this result, in \cite{SF} the system of equations is computed symbolically. One knows how to obtain $T_p$ and so the number of order condition equations. 
It blows up like the factorial with respect to the order $q$ (see Table \ref{tab:NumberOfTreesUpToOrder10}): this fact plays a key role in the choice of the evolutionary algorithm for the solution of the problem.
\begin{table}[htbp]
	\begin{center}
		\begin{tabular}{l||llllllllll}
			order 	   & 1 	& 2	& 3	& 4	& 5	& 6	& 7	& 8	& 9	& 10\\ \hline
			equations  & 1 	& 1	& 2	& 4	& 9 	& 20	& 48 	&115 	&286	&719
		\end{tabular}
	\end{center}

	\caption{Number of the order condition equations up to order $10$.}
	\label{tab:NumberOfTreesUpToOrder10}
\end{table}

We did not write the explicit form of $\Phi_j(t)$ and $\gamma(t)$ because these involve more notions about the rooted labelled trees (we refer to \cite{Hai} for more details). What we need to know, here, is that $\Phi_j(t)$ is written as a sum of $a_{i,j}$'s products and $\gamma(t)$ is an integer number coming from the tree's structure. 

For the reader do not use to this numerical analysis branch, we sketch some examples.

\begin{ex}
The parameters of a two levels explicit RKm have to satisfy 
\begin{equation*}
	\left\lbrace \begin{aligned}
	w_1 + w_2 &= 1;\\
	w_2a_{2,1} &=\frac{1}{2}.
	\end{aligned}\right. 
\end{equation*}
For instance the Euler method has $w_1=0$, $w_2=1$ and $a_{2,1}=1$.
\end{ex}

\begin{ex}
For an implicit three levels RKm of order three, the parameters fulfil
\begin{equation*}
	\left\lbrace \begin{aligned}
	\sum_{j=1}^3 w_j =&1;\\
	\sum_{j,k=1}^3 w_j a_{j,k} =&\frac{1}{2};\\
	\sum_{j,k,l=1}^3 w_j a_{j,k} a_{j,l}=&\frac{1}{3};\\
	\sum_{j,k,l=1}^3 w_j a_{j,k} a_{k,l}=&\frac{1}{6}.
	\end{aligned}\right. 
\end{equation*}
\end{ex}

The next theorem clarifies how the Taylor series of the numerical and the exact solution are. In fact, the $h^p$-terms of their series are expressed using $G^J(t)(y_0)$, the $J$-component of the elementary differential of $g$ corresponding to the tree $t$ evaluated at the point $y_0$. For every $G^J(t)(y_0)$ an integer coefficient $\alpha(t)$ appears naturally to weight its contribute in the Taylor sequence.

\begin{teo}\label{theo-local-error}
	If the Runge-Kutta method is of order $p$ and if $g$ is ($p+1$)- times continuously differentiable, we have
	\begin{equation}\label{sigma-local-error}
		y^J(y_0+h)-y_1^J=\frac{h^{p+1}}{(p+1)!}	\sum_{t\in T_{p+1}} \alpha(t)e(t) G^J(t)(y_0) + \vartheta(h^{p+2})
	\end{equation}
	where
	\begin{equation}\label{local-error}
		e(t)=1-\gamma(t)\sum_{j=1}^s w_j \Phi_j (t)
	\end{equation}
	is called the \emph{error coefficient} of the tree $t$.
\end{teo}

The error coefficient $e(t)$ expresses the difference of the Taylor series of the numerical and the exact solution at the term of the elementary differential given by $t$.

The optimization problem that we want to solve is minimizing the local approximation error of order $q+1$ in the set of RKms of order $q$: in other words, after finding feasible RKms of order $q$, our aim is to optimize such errors $e(t)$ for all the trees of order $q+1$.

\section{Runge-Kutta methods as an Algebraic Variety}\label{algebraic-variety}
In this section we present the problem of minimize the local error of a RKm as an optimization problem over an algebraic variety.
We define
\[
	V_{q}^s=\left\lbrace \mbox{$s$-level Runge-Kutta methods with accuracy order $q$}\right\rbrace 
\]
and we call $EV_{q}^s$ its subset having only the explicit ones.

Since the Table \ref{tab:BT}, the RKms have $s^2+2s$ free coefficients. The autonomous conditions (\ref{RK-condition-autonomo}) show that the parameters $c_i$ are dependent by $\{a_{i,j}\}$. 
Moreover, the order conditions in Theorem \ref{theo-order-condition} are polynomials, hence $V_{q}^s$ is an affine algebraic variety in the affine real space $\mathbb{A}^{s(s+1)}(\rea)$, minimally defined by the following polynomials in $s(s+1)$ variables:
\begin{equation*}
  \sum_{j=1}^s w_j \Phi_j(t)=\frac{1}{\gamma(t)}, \,\forall t\in T_1\cup T_2 \cup \dots \cup T_q.
\end{equation*}
Thus, we rewrite the Theorem \ref{theo-order-condition} as:
\begin{teo}\label{theo-variety} Let $\mathbf{x}=\left((a_{i,j})_{1\leq j,i \leq s} , (w_1, w_2, \dots, w_{s}) \right)$, then
\[
	\mbox{$\mathbf{x}$ is the parameters set of an $s$ levels RKm of order $q$} \Leftrightarrow \mathbf{x}\in V_{q}^s.
\]
\end{teo}
\begin{proof}
Considering $\mathbf{x}$ as the parameters of an $s$ levels RKm of order $q$, then $\mathbf{x}$ satisfies the order condition equations (\ref{constrain}) and then $\mathbf{x}$ belongs to $V_{q}^s$ as a point in the affine space $\mathbb{A}^{s(s+1)}(\rea)$. 
Vice-versa if the point $\mathbf{x}$ of $\mathbb{A}^{s(s+1)}(\rea)$ belongs to $V_{q}^s$, then $\mathbf{x}$ satisfies the equations (\ref{constrain}), that are also the order condition equations of the $s$-level RKms of order $q$.
\end{proof}

Similarly the algebraic variety $EV^s_q$ in $\mathbb{A}^{\frac{s(s+1)}{2}}(\rea)$, minimally defined by the same polynomial equations and by $\{a_{i,j}=0\,\, \forall i\geq j\}$, is the variety of the explicit $s$ levels RKms of order $q$; $EV^s_q$ is also a sub-variety of $V^s_q$. The Corollary \ref{theo-variety} holds too. 

\begin{cor}\label{cor-variety-ex}
Let $\mathbf{y}=\left((a_{i,j})_{1\leq i < j \leq s} , (w_1, w_2, \dots, w_{s}) \right)$, then
\[
	 \mbox{$\mathbf{y}$ is the parameters set of an $s$ levels explicit RKm of order $q$} \Leftrightarrow \mathbf{y}\in EV_{q}^s.
\]
\end{cor}

Avoiding repetitions of $EV_{q}^s$ and $V_{q}^s$, every results that we state for $EV_q^s$ is true also for $V_q^s$.

We know that the $EV_{q}^s$ is composite by connected component: $EV_{q}^s=\cup_{i\in I}U_i$. One knows that for large value of $s$ and $q$ the computational time for computing the local parameterizations,
\[
	\theta_{q, i}^s: \mathbb{A}^{d_i}(\rea) \rightarrow U_i,
\]
is not acceptable. In fact a symbolic approach of the problem is not feasible (look at \cite{SF}).

Since minimizing the local error over the variety $EV_q^s$ involves those parameterizations, then the use of the EAs is legitimate.

\section{{\sc Evo-Runge-Kutta}}\label{evo-runge-kutta}

The aim of this research work is to obtain new explicit or implicit Runge-Kutta methods of maximal order $q$ that minimizes local errors (\ref{local-error}) of order $q+1$. Now, We state the optimization problem for $EV_q^s$, because we will show the numerical results about $EV_3^3$, $EV_4^4$ and $EV_4^5$, but everything hold with the opportune modifications for $V_q^s$ too.

We denote, as in Corollary \ref{cor-variety-ex},
\[
	\mathbf{x}=\left((a_{i,j})_{1\leq i < j \leq s} , (w_1, w_2, \dots, w_s) \right)\in \rea^{\frac{s(s+1)}{2}},
\]
where $\{a_{i,j}\}$ and $\{w_i\}$ are the coefficients of the Butcher Tableau of an explicit Runge-Kutta method. $\mathbf{x}$ is a RKm of accuracy order $q$ (i.e. it lies in $EV_{q}^s$) if and only if it respects the following constrains:
\[
	\sum_{j=1}^s w_j \Phi_j(t)=\frac{1}{\gamma(t)},\, \forall t\in T_1\cup T_2 \cup \dots \cup T_q.
\]
Moreover, we want that $\mathbf{x}$ minimizes the local errors:
\[
	e(t)=1-\gamma(t)\sum_{j=1}^s w_j \Phi_j (t), \, \forall t\in T_{q+1}
\]
The impossibility of analytically computing the functions $\theta_{q,i}^s$ and the hardness of analyse it numerically gives us the motivation to use evolutionary algorithms to face this difficult global optimization problem (see \cite{goldberg:02}). 
Moreover, we produce a remarkable set of feasible solutions that could be used for different optimization problems over the algebraic varieties $EV_{q}^s$ and $V_{q}^s$: for examples, minimizing the computational time or maximizing the convergence area $S_a$ for the implicit RKms.

Fixed the number of levels $s$ and the order of accuracy $q$ the feasible conditions for a RKm being of order $p$ are given by
\begin{equation}\label{eq:feasible}
	\frac{\sum_{t\in T_p} \alpha(t)|e(t)|}{\sum_{t\in T_p} \alpha(t)}< c_p,\, \forall p=1,2,\dots, q
\end{equation}
where the $c_p$ values is showed in the Table \ref{tab:c_p}. 
We require an error large at most $1\cdot 10^{-15}$ for each constrain, but we use the \emph{amplification factor} $4$ to widen the tubular neighbourhood of the varieties $EV_q^s$ and to develop a greater number of acceptable solutions, keeping the \emph{diversity} of the population.

\begin{table}[htbp]
\begin{center}
\begin{tabular}{c|c|c}
	order 	&  Number of condition equations & $c_p$\\ \hline \hline
	1	&1	 &$1*4*10^{-15}$\\
	2	&2	 &$2*4*10^{-15}$\\
	3	&4	 &$4*4*10^{-15}$\\
	4	&8	 &$8*4*10^{-15}$\\
	5	&17	 &$17*4*10^{-15}$\\
	6	&37	 &$37*4*10^{-15}$\\
	7	&85	 &$85*4*10^{-15}$\\
	8	&200	 &$200*4*10^{-15}$\\
	9	&486	 &$486*4*10^{-15}$\\
	10	&1205	 &$1205*4*10^{-15}$
\end{tabular}
\caption{Values of $c_p$ for each order up to $10$}
\label{tab:c_p}
\end{center}
\end{table}

The fitness function is
\begin{equation}\label{eq:fitness}
	\mathcal{F}_q(\mathbf{x})=\frac{\sum_{i=0}^{q+1} \sum_{t\in T_i} \alpha(t)|e(t)|}{\sum_{i=0}^{q+1} \sum_{t\in T_i} \alpha(t)}.
\end{equation}

Summarizing, we narrow down the general algorithm {\sc Evo-Alg-Variety} (see Section \ref{sec-optimization-problem}) fixing $X=EV^s_q$ and $N=\nicefrac{s(s+1)}{2}$. The minimal basis of generators for $EV^s_q$ is given by (\ref{local-error}) and we subdivide them in $q$ sets, following the subdivision in labelled rooted trees. 
We also use the combinatorial definition of the Runge-Kutta order conditions to set the weight: in fact 
\[
  r_t=\frac{\alpha(t)}{\sum_{i=0}^{q+1} \sum_{t'\in T_i} \alpha(t')}
\]
is the weight of the equation $1-\gamma(t)\sum_{j=1}^s w_j \Phi_j (t)=0$, with $t\in T_p$, as suggested by (\ref{sigma-local-error}). 
The final setup is the definition of the $\epsilon_p=c_p$, concordly with Table \ref{tab:c_p}. In conclusion, $\mathcal{H}_q=\mathcal{F}_q$ and $\mathcal{G}=\mathcal{F}_{q+1}$ as in (\ref{eq:fitness}).

This particular choice of $\mathcal{H}_q$ and $\mathcal{G}$ simplifies the algorithm: {\sc Evo-Runge-Kutta} has no final optimization of $\mathcal{G}$, because $\mathcal{G}=\mathcal{H}_{q+1}$ and the minimizing problem of $\mathcal{G}$ is seen as another repetition of the main cycle. 

\begin{table}[htbp]
\begin{tabbing}
{\bf \sc Evo-Runge-Kutta}($s$, $q_{start}$, $\lambda$, $\mu$, $\{c_p\}_{p=1,\dots, q}$, $acc-mean$, $acc-best$, \\
\hspace{60pt}$I_{max}$, $best$, $oldSolutions$, others par's)\\
1. $output=true$;\\
2. $q=q_{start}$;\\
3. {\bf While}	\= (output)\\
4. 		\> output=Evo-Runge-Kutta-Cycle ($s$, $q$,...);\\
5.		\> q=q+1;\\
6. {\bf End While};\\
\end{tabbing}
\caption{The Algorithm Evo-Runge-Kutta}
\label{alg:evo-runge-kutta}
\end{table}

However, in the Runge Kutta case, we do not know the dimension of the varieties $EV_q^s$. This approach allows as to solve this complication. In fact, for any choice of the number of levels $s$, we always run the {\sc Evo-Runge-Kutta} with $q=2$. 
The algorithm explores the space of solutions, $\mathbb{A}^N_{\rea}$, finding points in $\epsilon(EV_2^s)$; after that, under the constrain being in $\epsilon(EV_2^s)$, it finds feasible solutions in $\epsilon(EV_3^s)$; and so on until the variety $\epsilon(EV_{q+1}^s)$ is empty. 
{\sc Evo-Runge-Kutta} cannot prove that $\epsilon(EV_{q+1}^s)$ is empty. What happens is that the algorithm minimizes automatically the local error of order $q+1$ in $\epsilon(EV_q^s)$ (i.e. the fitness function $\mathcal{G}$) and it fails to find new feasible solutions in $\epsilon(EV_{q+1}^s)$.

\begin{table}[htbp]
\begin{tabbing}
{\bf \sc Evo-Runge-Kutta-Cycle}($s$, $q$, $\lambda$, $\mu$, $\{c_p\}_{p=1,\dots, q}$, $acc-mean$, $acc-best$, \\
\hspace{60pt}$I_{max}$, $best$, $oldSolutions$, others par's)\\
1. t=0;\\
2. inizializePopulation($pop^{(t)}_d$);  /* random generation of RKs */ \\
3. initialize($newSolution$); /* new array of feasible solutions */\\
4. evaluationPopulation($pop^{(t)}_d$); /* evaluation of RK systems */   \\
5. {\bf While} \= ((t$<$ $I_{max}$)\&\&bestError$<$ $acc-best$)\\
6.       \>copy ($pop^{(t)}_d$, $popMut^{(t)}_d$); \\
7.       \>mutationOperator($popMut^{(t)}_d$, $\mu$);\\              
8.	 \>isFeasible($popMut^{(t)}_d$, $oldSolutions$);\\
9.       \>evaluationPopulation($popMut^{(t)}_d$);\\	
10.       \>$pop^{(t+1)}_d$=selection($pop^{(t)}_d$, $popMut^{(t)}_d$);\\
11.      \>computeStatistics($pop^{(t+1)}_d$);\\'
12.	 \>saveSolutions($newSolution$);\\	
13.     \>t=t+1;\\
14. {\bf End While};\\
15. {\bf If} 	\> ($newSolution$ is empty) {\bf Then} {\bf Return} false;\\	
16. {\bf Else} \> {\bf Return} true;\\
17. {\bf End If};
\end{tabbing}
\caption{The main cycle of the algorithm {\sc Evo-Runge-Kutta}}
\label{alg:evo-runge-kutta-cycle}
\end{table}

In Table \ref{alg:evo-runge-kutta} we show the structure of the algorithm: the main cycle is explained in Table \ref{alg:evo-runge-kutta-cycle}. 
Table \ref{tab.para} shows the parameters of the evolutionary algorithm.

\begin{table}[htbp]
\begin{center}
\begin{tabular}{l|l|c}
	Parameter & description							& Value\\ \hline \hline 
 	$q_{start}$ & starting value for the order the RK-methods			& fixed\\
	$s$ & levels of the RK-methods						& fixed\\
	$\lambda$ & number of elements of pop.					& $1000$\\
	$\mu$ & number of crossing over for pop.				& $500$\\
	$I_{max}$ & maximum iterations of the cycle 				& $10^5$\\
	$acc-mean$ & acceptable value for fitness's mean of the pop. 		& $c_{q+1}$\\
	$acc-best$ & acceptable value for fitness of the best element		& $\nicefrac{c_{q+1}}{4}$\\	
	$output$	& output of {\sc Evo-Runge-Kutta-cycle}			& true/false\\	
	$oldSolutions$  & points in $\epsilon(X_q)$				& \\
	$newSolutions$  & points in $\epsilon(X_{q+1})$				& 
\end{tabular}
\caption{Parameters of used in evolutionary algorithm {\sc Evo-Runge-Kutta}}
\label{tab.para}
\end{center}
\end{table}

We use the package {\sc CMA-ES} (see \cite{cma}) for the evolutionary optimization.

\section{The Numerical Results for $EV_3^3$, $EV_4^4$ and $EV_4^5$}\label{results}
In this section we show the numerical results obtained with {\sc Evo-Runge-Kutta} for the explicit case in the levels $s=3,4,5$. Since \cite{Butcher:08}, the variety $EV_5^5$ is empty and so the optimization for $s=5$ runs over the same equations of $s=4$ but with more variables; hence $EV_4^5$ should be dimensionally bigger than $EV_4^4$. The Table \ref{tab:feasible-point} shows the feasible points founded during the running: it respects what expected. 

\begin{table*}[htbp]
\begin{center}
\begin{tabular}{c|c|c|c}
  Order/ Level &        3&         4&         5\\ \hline
	      2&        46&        108&      34\\
	      3&        \textbf{146}&       197&      140\\
	      4&        0&         \textbf{364}&      \textbf{932}\\
	      5&        0&         0&         0
 \end{tabular}
\end{center}
\caption{Feasible solutions for the varieties $EV_q^s$ with $s=2,3,4$ and $q=2,3,4,5$}
\label{tab:feasible-point}
\end{table*}

In addiction, we present in Table \ref{tab:runge-kutta-3-3}, and respectively in Table \ref{tab:runge-kutta-4-4}, some of the Runge Kutta elements in Pareto front of the final optimization in $EV_3^3$, and respectively $EV_4^4$: the condition $e_{0,0}=e(t_{0,0})$ is zero for both these computations (we enumerate the trees and their local error as in \cite{Hai}).

To have a proper view of the results, we list the Runge Kutta methods suggested in the Table \ref{tab:runge-kutta-4-4} with all the digits from the Table \ref{tab:rk.primo} to \ref{tab:rk.ultimo}.

\begin{table*}[htbp]
\begin{center}
\begin{scriptsize}
\begin{tabular}{cccc}
  0.6284799329301066 &&&\\
  0.0768995143272878 & 0.2946205527425946 &&\\
  0.9213018835278037 & -0.3307380161751250 & 0.4498926859290699 &\\ \hline
  0.1599508127125725 &  0.2455227275515280 &     0.4335530269375665 &    0.1609734327983330 \\
\end{tabular}
\end{scriptsize}
\caption{The Butcher tableau of first $4$ levels Runge-Kutta in Table \ref{tab:runge-kutta-4-4}}
\label{tab:rk.primo}
\end{center}
\end{table*}
 
\begin{table*}[htbp]
\begin{center}
\begin{scriptsize}
\begin{tabular}{cccc}
0.7363018963226102 &&&\\
-0.0824111295277597 &    0.3461092332051678 &&\\
0.8341073930886564 & -0.2376248889779167 & 0.4318380713817837 &\\ \hline
0.1571049761428195 &     0.2866708210526255 &     0.3980679473446672 &     0.1581562554598879\\
\end{tabular}
\end{scriptsize}
\caption{The Butcher tableau of second $4$ levels Runge-Kutta in Table \ref{tab:runge-kutta-4-4}}
\end{center}
\end{table*}

\begin{table*}[htbp]
\begin{center}
\begin{scriptsize}
\begin{tabular}{cccc}
0.7066207865932520 &&&\\
-0.1986791412551434 &  0.4920583546618451 &&\\
0.9521842711595797 & -0.3227487964624350 &  0.4140057189110243 &\\ \hline
0.1495805086972162  &    0.2852315575222697   &   0.4127591061897394  &   0.1524288275907748 \\
\end{tabular}
\end{scriptsize}
\caption{The Butcher tableau of third $4$ levels Runge-Kutta in Table \ref{tab:runge-kutta-4-4}}
\end{center}
\end{table*}

\begin{table*}[htbp]
\begin{center}
\begin{scriptsize}
\begin{tabular}{cccc}
0.7559397531895379 &&&\\
0.0103807700662207 &     0.2336794767442519 &&\\
0.7531440557235027 &     -0.1855193036061356 &    0.4515825081294965 &\\ \hline
0.1620639462265475  &    0.2833431842164397 &     0.3921840325526135 &    0.1624088370043993 \\
\end{tabular}
\end{scriptsize}
\caption{The Butcher tableau of $4$-th $4$ levels Runge-Kutta in Table \ref{tab:runge-kutta-4-4}}
\end{center}
\end{table*}

\begin{table*}[htbp]
\begin{center}
\begin{scriptsize}
\begin{tabular}{cccc}
0.7963012365770038 &&&\\
-0.3628863293699752 & 0.5665850927929587 &&\\
0.9145513895258300 & -0.2798774916728459 & 0.3940759984502829 &\\ \hline
0.1451855083814838  &    0.3164608254865673  &    0.3906989174541415  &    0.1476547486778075\\
\end{tabular}
\end{scriptsize}
\caption{The Butcher tableau of $5$-th $4$ levels Runge-Kutta in Table \ref{tab:runge-kutta-4-4}}
\end{center}
\end{table*}

\begin{table*}[htbp]
\begin{center}
\begin{scriptsize}
\begin{tabular}{cccc}
0.7573881838384795 &&&\\
-0.1741476827373300 & 0.4167594988988161 &&\\
0.8573072494974391 & -0.2469806098316732 & 0.4182913673468296 &\\ \hline
0.1534104115563865  &    0.2981759720267748  &    0.3935172041716111 &    0.1548964122452277\\
\end{tabular}
\end{scriptsize}
\caption{The Butcher tableau of $6$-th $4$ levels Runge-Kutta in Table \ref{tab:runge-kutta-4-4}}
\end{center}
\end{table*}

\begin{table*}[htbp]
\begin{center}
\begin{scriptsize}
\begin{tabular}{cccc}
0.6144690308782115 &&&\\
0.0769053801528305 & 0.3086255889689522 &&\\
0.9475038121348336 & -0.3525125955836914 & 0.4489168429216662 &\\ \hline
0.1594196652643455 &  0.2418674636726206 & 0.4381018921031239 & 0.1606109789599099\\
\end{tabular}
\end{scriptsize}
\caption{The Butcher tableau of $7$-th $4$ levels Runge-Kutta in Table \ref{tab:runge-kutta-4-4}}
\end{center}
\end{table*}

\begin{table*}[htbp]
\begin{center}
\begin{scriptsize}
\begin{tabular}{cccc}
0.6957301007596283 &&&\\
0.0048639630817450 &0.2994059361585764 &&\\	
0.8469017267939670 & -0.2590107629375511 & 0.4434995354305374 &\\ \hline
0.1594489916633499 & 0.2691847382771309 & 0.4110483067778322 & 0.1603179632816870\\
\end{tabular}
\end{scriptsize}
\caption{The Butcher tableau of $8$-th $4$ levels Runge-Kutta in Table \ref{tab:runge-kutta-4-4}}
\end{center}
\end{table*}

\begin{table*}[htbp]
\begin{center}
\begin{scriptsize}
\begin{tabular}{cccc}
0.6313830140634179 &&&\\
-0.0266776618862919 & 0.3952946478228191 &&\\
0.9776608341341720 & -0.3627250371329542 &0.4349091893706292 &\\ \hline
0.1552061988645526 & 0.2544888797173599 & 0.4329627509848846 & 0.1573421704332030\\
\end{tabular}
\end{scriptsize}
\caption{The Butcher tableau of $9$-th $4$ levels Runge-Kutta in Table \ref{tab:runge-kutta-4-4}}
\label{tab:rk.ultimo}
\end{center}
\end{table*}

\begin{sidewaystable}[htbp]
\begin{scriptsize}
\begin{center}
\begin{tabular}{cccccc|ccc}
$a_{2,1}$ &$a_{3,1}$  &$a_{3,2}$     &$w_1$        &$w_2$                &$w_3$        &$e_{1,0}$  &$e_{2,1}$  &$e_{2,2}$\\
0.7451612195870713 & -0.0009391259740175 &   0.5448300494088465 & 0.4105229605812859 & 0.3569548761324413 & 0.2325221632862728 & 4  & 0 & 11 \\
0.8359905251670485 & -0.0664387342696503 &    0.4584899731014609 & 0.4348280500453142 &    0.3606999172446620 & 0.2044720327100238 & 2  & 7  & 0\\
0.6755587558980997& 0.0455901340664806 &  0.5803039629935742 & 0.4251381574145029 &     0.3332944491731594 &     0.2415673934123377 & 6 & 4 & 4\\
0.6665035492138531& 0.0600660735594445 & 0.5500843178130679 & 0.4545870076958900 &     0.3085179558934952 & 0.2368950364106148 &4 &9 & 1\\
0.8270914980689880& -0.0584896696476523 & 0.5046425166264479 & 0.3993110749282756  &    0.3826248707541401  &    0.2180640543175842 &2&3&5\\
0.6023264657110801& 0.1075766214853408 &   0.5716638779026584 & 0.4840342046290478 &  0.2735569446838536 & 0.2424088506870987 &2 & 9& 3
\end{tabular}
\end{center}
\caption{The Butcher tableau of the new solutions for a $3$ levels Runge-Kutta of order $3$; the error coefficients are of order $10^{-16}$.}
\label{tab:runge-kutta-3-3}
\end{scriptsize}
\end{sidewaystable}

\begin{sidewaystable}[htbp]
\begin{center}
\begin{tabular}{cccccccccc|ccccccc}
$a_{2,1}$ &$a_{3,1}$        &$a_{3,2}$ &$a_{4,1}$ &$a_{4,2}$        &$a_{4,3}$ &$w_1$ &$w_2$ &$w_3$ &$w_4$        &$e_{1,0}$ &$e_{2,1}$ &$e_{2,2}$ &$e_{3,0}$ &$e_{3,1}$ &$e_{3,2}$ &$e_{3,3}$\\
0.628 & 0.076 & 0.294 & 0.921 & -0.330 & 0.449 & 0.159 &     0.245 &     0.433 &    0.160 & 6 & 1 & 4 & 11 & 8 & 4 & 27\\
0.736 &-0.082 &    0.346 & 0.834 & -0.237 & 0.431 &0.157 &     0.286 &     0.398 &     0.158 & 8 &5 & 4 & 13 & 3 & 9 & 14\\
0.706 & -0.198 &  0.492 & 0.952 & -0.322 &  0.414 &  0.149  &    0.285   &   0.412  &   0.152 & 1 & 4 &4 & 1 & 12 & 20 & 05 \\
0.755 &  0.010 &     0.233 & 0.753 &     -0.185 &    0.451 & 0.162  &    0.283 &     0.392  &    0.162 & 6 & 2 & 4 & 2 & 1 & 5 & 54\\
0.796  & -0.362 & 0.566  & 0.914  & -0.279 & 0.394 & 0.145  &    0.316  &    0.390  &    0.147  & 0 & 3 & 10 & 3 & 0 & 4 & 10\\
0.757  & -0.174  & 0.416  & 0.857  & -0.246  & 0.418  & 0.153   &    0.298   &    0.393  &    0.154 & 8 & 15 & 7 & 3 & 0 & 0 & 5\\
0.614 & 0.076 & 0.308 & 0.947 & -0.352 & 0.448 & 0.159 &  0.241 & 0.438 & 0.160  & 8 & 6 & 2 & 3 & 9 & 7 & 14\\
0.695 & 0.004  &0.299  & 0.846  & -0.259  & 0.443  & 0.159 & 0.269 & 0.411 & 0.160  & 9 & 1 & 9 & 5 & 1 & 31 & 4 \\
0.631  & -0.026  & 0.395  & 0.977  & -0.362 & t0.434  & 0.155  & 0.254  & 0.432  & 0.157  &4 & 15 & 9 & 2 & 5 & 19 & 7
\end{tabular}
\end{center}
\caption{The Butcher tableau of the new solutions for a $4$ levels Runge-Kutta of order $4$; the error coefficients are of order $10^{-16}$.}
\label{tab:runge-kutta-4-4}
\end{sidewaystable}

\section{Conclusion}
The designed and implemented evolutionary algorithm, {\sc Evo-Runge-Kutta}, optimizes implicit or explicit Runge-Kutta methods in order to find the maximal order of accuracy and to minimize theirs local error in the next order.
The results presented in this article suggest that further work in this research field will advance the designing of Runge-Kutta methods, in particular, and the use of the Evolutionary Algorithm for any kind of optimization over an algebraic variety.
To our knowledge this is the first time that algebraic geometry is used to state correctly that evolutionary algorithms have to be used to face a particular optimization problem. 
Again we think this is the first time that algebraic geometry and evolutionary algorithms are used to tackle a numerical analysis problem. Further refinement of our evolutionary optimization algorithm will surely improve the solution of these important numerical analysis problem. 

\vskip0.3cm

A copy of the software can be obtained by sending an email to the authors. 


\end{document}